\documentclass[12pt]{article}
\usepackage{amsmath,amsfonts,amssymb,amsthm,txfonts,hyperref,url,t1enc}
\newtheorem{theorem}{Theorem}
\newtheorem{lemma}{Lemma}
\newtheorem{observation}{Observation}
\newcommand{\N}{\mathbb N}
\newcommand{\Z}{\mathbb Z}
\newcommand{\K}{\textbf{\textit{K}}}
\begin{document}
\noindent
\centerline{{\large Does there exist an algorithm which to each Diophantine equation}}
\vskip 0.2truecm
\noindent
\centerline{{\large assigns an integer which is greater than the number (heights)}}
\vskip 0.2truecm
\noindent
\centerline{{\large of integer solutions, if these solutions form a finite set?}}
\vskip 0.6truecm
\noindent
\centerline{{\large Apoloniusz Tyszka}}
\vskip 0.6truecm
\begin{sloppypar}
\noindent
{\bf Abstract.} Let \mbox{$E_n=\{x_i=1,~x_i+x_j=x_k,~x_i \cdot x_j=x_k: i,j,k \in \{1,\ldots,n\}\}$}.
If Matiyasevich's conjecture on finite-fold Diophantine representations is true, then for every
computable function \mbox{$f:\N \to \N$} there is a positive integer $m(f)$ such that for
each integer \mbox{$n \geq m(f)$} there exists a system \mbox{$S \subseteq E_n$} which has
at least $f(n)$ and at most finitely many solutions in integers \mbox{$x_1,\ldots,x_n$}.
This conclusion contradicts to the author's conjecture on integer arithmetic,
which implies that the heights of integer solutions to a Diophantine equation are computably bounded,
if these solutions form a finite set.
\end{sloppypar}
\vskip 0.6truecm
\noindent
{\large 1.~~Introduction}
\vskip 0.4truecm
\par
The heights of integer solutions to a Diophantine equation
\begin{equation}\label{equ1}
ax^2+bxy+cy^2+dx+ey+f=0
\end{equation}
are bounded from above by $20 \cdot \left({\rm max}(|a|,|b|,|c|,|d|,|e|,|f|)\right)^4$,
if equation~(\ref{equ1}) has at most finitely many integer solutions, see \cite{Schinzel},
\mbox{\cite[p.~17,~Theorem~1.14]{Narkiewicz}}, and \cite{Niven}.
Every Diophantine equation of degree at most $n$ has the form
\begin{equation}\label{equ2}
\sum_{\stackrel{\textstyle i_1,\ldots,i_k \in \{0,\ldots,n\}}{\textstyle i_1+\ldots+i_k \leq n}}
a(i_1,\ldots,i_k) \cdot x_1^{\textstyle i_1} \cdot \ldots \cdot x_k^{\textstyle i_k}=0
\end{equation}
where $a(i_1,\ldots,i_k)$ denote integers.
\vskip 0.2truecm
\begin{observation}\label{obs1}
For each positive integer $b$, there are at most finitely many equations~(\ref{equ2}) which satisfy
\[
{\rm max}\Bigl(\{k,n\} \cup \Bigl\{|a(i_1,\ldots,i_k)|:~\bigl(i_1,\ldots,i_k \in \{0,\ldots,n\}\bigr) \wedge \bigl(i_1+\ldots+i_k \leq n\bigr)\Bigr\}\Bigl) \leq b
\]
\end{observation}
We state the following double problem: {\em Does there exist a computable function of
\[
{\rm max}\Bigl(\{k,n\} \cup \Bigl\{|a(i_1,\ldots,i_k)|:~\bigl(i_1,\ldots,i_k \in \{0,\ldots,n\}\bigr) \wedge \bigl(i_1+\ldots+i_k \leq n\bigr)\Bigr\}\Bigl)
\]
which bounds the number (heights) of integer solution to equation~(\ref{equ2}), if these solutions form a finite set?}
\vskip 0.2truecm
\par
The existence of such bounds is discussed in this article. By Observation~\ref{obs1}, the stated problem is equivalent to the
double problem from the title of the article.
\vskip 0.4truecm
\noindent
{\large 2.~~Small systems of Diophantine equations with a large number of integer solutions}
\vskip 0.4truecm
\par
Let $E_n=\{x_i=1,~x_i+x_j=x_k,~x_i \cdot x_j=x_k: i,j,k \in \{1,\ldots,n\}\}$. The following system
\begin{displaymath}
\left\{
\begin{array}{rcl}
x_1 \cdot x_1 &=& x_1 \\
&\ldots& \\
x_n \cdot x_n &=&x_n
\end{array}
\right.
\end{displaymath}
has exactly $2^n$ solutions in integers $x_1,\ldots,x_n$. If $n \geq 10$, then $1156 \cdot 2^{n-10}>2^n$
and there is a simply defined system $S \subseteq E_n$ which has exactly $1156 \cdot 2^{n-10}$ solutions in integers
$x_1,\ldots,x_n$, see \cite{Browkin}. We strengthen this results assuming an old conjecture due to Yu. Matiyasevich.
\vskip 0.2truecm
\par
The Davis-Putnam-Robinson-Matiyasevich theorem states that every recursively enumerable
set \mbox{${\cal M} \subseteq {\N}^n$} has a Diophantine representation, that is
\begin{equation}
\tag*{\tt (R)}
(a_1,\ldots,a_n) \in {\cal M} \Longleftrightarrow
\exists x_1, \ldots, x_m \in \N ~~W(a_1,\ldots,a_n,x_1,\ldots,x_m)=0
\end{equation}
for some polynomial $W$ with integer coefficients, see \cite{Matiyasevich1} and \cite{Kuijer}.
The polynomial~$W$ can be computed, if we know a Turing machine~$M$
such that, for all \mbox{$(a_1,\ldots,a_n) \in {\N}^n$}, $M$ halts on \mbox{$(a_1,\ldots,a_n)$} if and only if
\mbox{$(a_1,\ldots,a_n) \in {\cal M}$}, see \cite{Matiyasevich1} and \cite{Kuijer}.
\newpage
\begin{sloppypar}
The representation~{\tt (R)} is said to be \mbox{finite-fold} if for any \mbox{$a_1,\ldots,a_n \in \N$} the equation
\mbox{$W(a_1,\ldots,a_n,x_1,\ldots,x_m)=0$} has at most finitely many solutions \mbox{$(x_1,\ldots,x_m) \in {\N}^m$}.
Yu.~Matiyasevich conjectures that each recursively enumerable set \mbox{${\cal M} \subseteq {\N}^n$}
has a \mbox{finite-fold} Diophantine representation, see \mbox{\cite[pp.~341--342]{Davis}},
\mbox{\cite[p.~42]{Matiyasevich2}} and \mbox{\cite[p.~79]{Matiyasevich3}}.
\end{sloppypar}
\vskip 0.2truecm
\par
Before the main Theorem~\ref{the1}, we need an algebraic lemma together with introductory matter.
\vskip 0.2truecm
\begin{sloppypar}
Let \mbox{$D(x_1,\ldots,x_p) \in {\Z}[x_1,\ldots,x_p]$}. For the Diophantine equation \mbox{$2 \cdot D(x_1,\ldots,x_p)=0$},
let $M$ denote the maximum of the absolute values of its coefficients.
Let ${\cal T}$ denote the family of all polynomials
$W(x_1,\ldots,x_p) \in {\Z}[x_1,\ldots,x_p]$ whose all coefficients belong to the interval $[-M,M]$
and ${\rm deg}(W,x_i) \leq d_i={\rm deg}(D,x_i)$ for each $i \in \{1,\ldots,p\}$.
Here we consider the degrees of $W(x_1,\ldots,x_p)$ and $D(x_1,\ldots,x_p)$
with respect to the variable~$x_i$.
\vskip 0.2truecm
\par
We choose any bijection \mbox{$\tau: \{p+1,\ldots,{\rm card}({\cal T})\} \longrightarrow {\cal T} \setminus \{x_1,\ldots,x_p\}$}.
Let ${\cal H}$ denote the family of all equations of the forms
\vskip 0.2truecm
\noindent
\centerline{$x_i=1$, $x_i+x_j=x_k$, $x_i \cdot x_j=x_k$~~($i,j,k \in \{1,\ldots,{\rm card}({\cal T})\})$}
\vskip 0.2truecm
\noindent
which are polynomial identities in \mbox{${\Z}[x_1,\ldots,x_p]$} if
\[
\forall s \in \{p+1,\ldots,{\rm card}({\cal T})\} ~~x_s=\tau(s)
\]
There is a unique \mbox{$q \in \{p+1,\ldots,{\rm card}({\cal T})\}$} such that \mbox{$\tau(q)=2 \cdot D(x_1,\ldots,x_p)$}.
For each ring $\K$ extending $\Z$ the system ${\cal H}$ implies \mbox{$2 \cdot D(x_1,\ldots,x_p)=x_q$}.
To see this, we observe that there exist pairwise distinct
\mbox{$t_0,\ldots,t_m \in {\cal T}$} such that $m>p$ and
\[
t_0=1~ \wedge ~t_1=x_1~ \wedge ~\ldots~ \wedge ~t_p=x_p~ \wedge ~t_m=2 \cdot D(x_1,\ldots,x_p)~ \wedge
\]
\[
\forall i \in \{p+1,\ldots,m\}~ \exists j,k \in \{0,\ldots,i-1\} ~~(t_j+t_k=t_i \vee t_i+t_k=t_j \vee t_j \cdot t_k=t_i)
\]
For each ring $\K$ extending $\Z$ and for each \mbox{$x_1,\ldots,x_p \in \K$}
there exists a unique tuple \mbox{($x_{p+1},\ldots,x_{{\rm card}({\cal T})}) \in \K^{{\rm card}({\cal T})-p}$}
such that the tuple \mbox{$(x_1,\ldots,x_p,x_{p+1},\ldots,x_{{\rm card}({\cal T})})$}
solves the system \mbox{${\cal H}$}. The sought elements \mbox{$x_{p+1},\ldots,x_{{\rm card}({\cal T})}$}
are given by the formula
\[
\forall s \in \{p+1,\ldots,{\rm card}({\cal T})\} ~~x_s=\tau(s)(x_1,\ldots,x_p)
\]
\newpage
\begin{lemma}\label{lem1}
The system ${\cal H} \cup \{x_q+x_q=x_q\}$ can be simply computed.
For each ring $\K$ extending $\Z$, the equation $D(x_1,\ldots,x_p)=0$
is equivalent to the system ${\cal H} \cup \{x_q+x_q=x_q\} \subseteq E_{{\rm card}({\cal T})}$.
Formally, this equivalence can be written as
\[
\forall x_1,\ldots,x_p \in \K ~\Bigl(D(x_1,\ldots,x_p)=0 \Longleftrightarrow
\exists x_{p+1},\ldots,x_{{\rm card}({\cal T})} \in \K
\]
\[
(x_1,\ldots,x_p,x_{p+1},\ldots,x_{{\rm card}({\cal T})}) {\rm ~solves~the~system~}
{\cal H} \cup \{x_q+x_q=x_q\} \Bigr)
\]
For each ring $\K$ extending $\Z$ and for each \mbox{$x_1,\ldots,x_p \in \K$} with
\mbox{$D(x_1,\ldots,x_p)=0$} there exists a unique tuple
\mbox{($x_{p+1},\ldots,x_{{\rm card}({\cal T})}) \in \K^{{\rm card}({\cal T})-p}$} such
that the tuple \mbox{$(x_1,\ldots,x_p,x_{p+1},\ldots,x_{{\rm card}({\cal T})})$} solves the system
\mbox{${\cal H} \cup \{x_q+x_q=x_q\}$}. Hence, for each ring $\K$ extending $\Z$ the equation
\mbox{$D(x_1,\ldots,x_p)=0$} has the same number of solutions as the system \mbox{${\cal H} \cup \{x_q+x_q=x_q\}$}.
\end{lemma}
\end{sloppypar}
\par
Putting $M=M/2$ we obtain new families ${\cal T}$ and ${\cal H}$.
There is a unique $q \in \{1,\ldots,{\rm card}({\cal T})\}$ such that
\[
\Bigl(q \in \{1,\ldots,p\}~ \wedge ~x_q=D(x_1,\ldots,x_p)\Bigr)~ \vee
\]
\[
\Bigl(q \in \{p+1,\ldots,{\rm card}({\cal T})\}~ \wedge ~\tau(q)=D(x_1,\ldots,x_p)\Bigr)
\]
The new system \mbox{${\cal H} \cup \{x_q+x_q=x_q\}$} is equivalent to \mbox{$D(x_1,\ldots,x_p)=0$}
and can be simply computed.
\begin{theorem}\label{the1}
If Matiyasevich's conjecture is true, then for every
computable function \mbox{$f:\N \to \N$} there is a positive integer $m(f)$ such that for
each integer \mbox{$n \geq m(f)$} there exists a system \mbox{$S \subseteq E_n$} which has
at least $f(n)$ and at most finitely many solutions in integers \mbox{$x_1,\ldots,x_n$}.
\end{theorem}
\begin{proof}
\begin{sloppypar}
\noindent
By Matiyasevich's conjecture, the function $\N \ni n \to f(n)! \in \N$ has a finite-fold Diophantine
representation. It means that there is a polynomial $W(x_1,x_2,x_3,\ldots,x_r)$ with integer coefficients such that
for each non-negative integers $x_1$, $x_2$,
\end{sloppypar}
\begin{equation}
\tag*{\tt (E1)}
x_1=f(x_2)! \Longleftrightarrow \exists x_3, \ldots, x_r \in \N ~~W(x_1,x_2,x_3,\ldots,x_r)=0
\end{equation}
and
\[
{\rm only~finitely~many~tuples~}(x_3,\ldots,x_r) \in {\N}^{r-2} {\rm ~satisfy~} W(x_1,x_2,x_3,\ldots,x_r)=0~~~~{\tt (A)}.
\]
By the equivalence~{\rm (E1)} and Lagrange's four-square theorem, for each integers $x_1$, $x_2$,
the conjunction \mbox{$(x_2 \geq 0) \wedge (x_1=f(x_2)!)$} holds true if and only if there exist integers
\[
a,b,c,d,\alpha,\beta,\gamma,\delta,x_3,x_{3,1},x_{3,2},x_{3,3},x_{3,4},\ldots,x_r,x_{r,1},x_{r,2},x_{r,3},x_{r,4}
\]
such that
\[
W^2(x_1,x_2,x_3,\ldots,x_r)+\bigl(x_1-a^2-b^2-c^2-d^2\bigr)^2+\bigl(x_2-\alpha^2-\beta^2-\gamma^2-\delta^2\bigr)^2+
\]
\[
\bigl(x_3-x^2_{3,1}-x^2_{3,2}-x^2_{3,3}-x^2_{3,4}\bigr)^2+\ldots+\bigl(x_r-x^2_{r,1}-x^2_{r,2}-x^2_{r,3}-x^2_{r,4}\bigr)^2=0
\]
The sentence~{\tt (A)} guarantees that for each integers $x_1$, $x_2$, only finitely many integer tuples
\[
\bigl(a,b,c,d,\alpha,\beta,\gamma,\delta,x_3,x_{3,1},x_{3,2},x_{3,3},x_{3,4},\ldots,x_r,x_{r,1},x_{r,2},x_{r,3},x_{r,4}\bigr)
\]
satisfy the last equality. By Lemma~\ref{lem1}, there is an integer \mbox{$s \geq 3$} such that for each integers $x_1$, $x_2$,
\begin{equation}
\tag*{\tt (E2)}
\Bigl(x_2 \geq 0 \wedge x_1=f(x_2)!\Bigr) \Longleftrightarrow \exists x_3,\ldots,x_s \in \Z ~~\Psi(x_1,x_2,x_3,\ldots,x_s)
\end{equation}
where the formula $\Psi(x_1,x_2,x_3,\ldots,x_s)$ is algorithmically determined as a conjunction of formulae of the forms
\mbox{$x_i=1$}, \mbox{$x_i+x_j=x_k$}, \mbox{$x_i \cdot x_j=x_k$} \mbox{($i,j,k \in \{1,\ldots,s\}$)} and
for each integers \mbox{$x_1,x_2$} at most finitely many integer tuples \mbox{$(x_3,\ldots,x_s)$} satisfy
\mbox{$\Psi(x_1,x_2,x_3,\ldots,x_s)$}.
Let $m(f)=8+2s$, and let $[\cdot]$ denote the integer part function. For each integer $n \geq m(f)$,
\[
n-\left[\frac{n}{2}\right]-4-s \geq m(f)-\left[\frac{m(f)}{2}\right]-4-s \geq m(f)-\frac{m(f)}{2}-4-s=0
\]
Let $S$ denote the following system
\[\left\{
\begin{array}{rcl}
{\rm all~equations~occurring~in~}\Psi(x_1,x_2,x_3,\ldots,x_s) \\
n-\left[\frac{n}{2}\right]-4-s {\rm ~equations~of~the~form~} z_i=1 \\
t_1 &=& 1 \\
t_1+t_1 &=& t_2 \\
t_2+t_1 &=& t_3 \\
&\ldots& \\
t_{\left[\frac{n}{2}\right]-1}+t_1 &=& t_{\left[\frac{n}{2}\right]} \\
t_{\left[\frac{n}{2}\right]}+t_{\left[\frac{n}{2}\right]} &=& w \\
w+y &=& x_2 \\
y+y &=& y {\rm ~(if~}n{\rm ~is~even)} \\
y &=& 1 {\rm ~(if~}n{\rm ~is~odd)} \\
u \cdot v &=& x_1
\end{array}
\right.\]
with $n$ variables. By the equivalence~{\tt (E2)}, the system~$S$ is consistent over $\Z$.
If an integer $n$-tuple $(x_1,x_2,x_3,\ldots,x_s,\ldots,w,y,u,v)$ solves~$S$,
then by the equivalence~{\tt (E2)},
\[
x_1=f(x_2)!=f(w+y)!=f\left(2 \cdot \left[\frac{n}{2}\right]+y\right)!=f(n)!
\]
If $f(n)=0$, then the equation \mbox{$u \cdot v=x_1=f(n)!=1$} has at least $f(n)$ and at most
finitely many solutions in integers \mbox{$u,v$}.
If $f(n) \geq 1$ and \mbox{$u \in \{1,\ldots,f(n)\}$}, then $u$ divides $f(n)!$.
Hence, the equation \mbox{$u \cdot v=x_1=f(n)!$} has at least $f(n)$ and at most finitely many
solutions in integers \mbox{$u,v$}. In both cases, the conclusion transfers to integer solutions of $S$.
\end{proof}
\vskip 0.2truecm
\par
If we do not assume Matiyasevich's conjecture, then the system~$S$
is still consistent over~$\Z$, but may have infinitely many integer solutions. Always,
if an integer \mbox{$n$-tuple} \mbox{$(x_1,x_2,x_3,\ldots,x_s,\ldots,w,y,u,v)$} solves~$S$, then
\mbox{$x_1=f(n)!$}. By choosing a rapidly growing function \mbox{$f:\N \to \N$},
we can guarantee that each integer solution of~$S$ is very large.
\vskip 0.4truecm
\noindent
{\large 3.~~Matiyasevich's conjecture vs the author's conjecture on integer arithmetic}
\vskip 0.4truecm
\par
Matiyasevich's conjecture remains in contradiction to the following Conjecture due to the author, see Theorem~\ref{the2}.
\vskip 0.2truecm
\noindent
{\bf Conjecture}~(\cite{Cipu},~\cite{Tyszka}). {\em If a system $S \subseteq E_n$ has only finitely many solutions in integers \mbox{$x_1,\ldots,x_n$},
then each such solution \mbox{$(x_1,\ldots,x_n)$} satisfies \mbox{$|x_1|,\ldots,|x_n| \leq 2^{\textstyle 2^{n-1}}$}.}
\begin{observation}\label{obs2}
For $n \geq 2$, the bound $2^{\textstyle 2^{n-1}}$ cannot be decreased because the system
\begin{displaymath}
\left\{
\begin{array}{rcl}
x_1+x_1 &=& x_2 \\
x_1 \cdot x_1 &=& x_2 \\
x_2 \cdot x_2 &=& x_3 \\
x_3 \cdot x_3 &=& x_4 \\
&\ldots& \\
x_{n-1} \cdot x_{n-1} &=&x_n
\end{array}
\right.
\end{displaymath}
\noindent
has exactly two integer solutions, namely $(0,\ldots,0)$ and
$\Bigl(2,4,16,256,\ldots,2^{\textstyle 2^{n-2}},2^{\textstyle 2^{n-1}}\Bigr)$.
\end{observation}
\newpage
\begin{sloppypar}
\begin{theorem}\label{the2}
The Conjecture formulated for an arbitrary computable bound
\mbox{$\beta: \N \setminus \{0\} \to \N$} instead of the bound \mbox{$\N \setminus \{0\} \ni n \to 2^{\textstyle 2^{n-1}} \in \N$}
remains in contradiction to Matiyasevich's conjecture.
\end{theorem}
\end{sloppypar}
\begin{proof}
Assume that the reformulated Conjecture is true. Then, if a system \mbox{$S \subseteq E_n$} has only finitely many solutions
in integers \mbox{$x_1,\ldots,x_n$}, then the number of solutions does not exceed \mbox{$\left(1+2 \cdot \beta(n)\right)^n$}.
Assume that Matiyasevich's conjecture is true. By applying Theorem~\ref{the1} for \mbox{$f(n)=\left(1+2 \cdot \beta(n)\right)^n+1$},
we conclude that for a sufficiently large value of~$n$, there is a system \mbox{$S \subseteq E_n$} which has at least
\mbox{$\left(1+2 \cdot \beta(n)\right)^n+1$} and at most finitely many solutions in integers \mbox{$x_1,\ldots,x_n$},
a contradiction.
\end{proof}
\vskip 0.4truecm
\noindent
{\large 4.~~On the author's conjecture}
\vskip 0.4truecm
\par
The Conjecture implies that if equation~(\ref{equ2}) has only finitely many solutions in integers
(non-negative integers, rationals), then their heights are bounded from above by
a computable function of
\[
{\rm max}\Bigl(\{k,n\} \cup \Bigl\{|a(i_1,\ldots,i_k)|:~\bigl(i_1,\ldots,i_k \in \{0,\ldots,n\}\bigr) \wedge \bigl(i_1+\ldots+i_k \leq n\bigr)\Bigr\}\Bigl)
\]
see \cite{Tyszka}. 
\vskip 0.2truecm
\par
To each system $S \subseteq E_n$ we assign the system $\widetilde{S}$ defined by
\vskip 0.2truecm
\par
\noindent
\centerline{$\left(S \setminus \{x_i=1:~i \in \{1,\ldots,n\}\}\right) \cup$}
\par
\noindent
\centerline{$\{x_i \cdot x_j=x_j:~i,j \in \{1,\ldots,n\} {\rm ~and~the~equation~} x_i=1 {\rm ~belongs~to~} S\}$}
\vskip 0.2truecm
\par
\noindent
In other words, in order to obtain $\widetilde{S}$ we remove from $S$ each
equation $x_i=1$ and replace it by the following $n$ equations:
\vskip 0.2truecm
\par
\noindent
\centerline{$\begin{array}{rcl}
x_i \cdot x_1 &=& x_1\\
&\ldots& \\
x_i \cdot x_n &=& x_n
\end{array}$}
\begin{lemma}\label{lem2}
For each system $S \subseteq E_n$
\begin{eqnarray*}
\{(x_1,\ldots,x_n) \in {\Z}^n:~(x_1,\ldots,x_n) {\rm ~solves~} \widetilde{S}\} &=& \\
\{(x_1,\ldots,x_n) \in {\Z}^n:~(x_1,\ldots,x_n) {\rm ~solves~} S\} \cup
\{(0,\ldots,0)\}&
\end{eqnarray*}
\end{lemma}
\newpage
\noindent
{\bf Corollary.} {\em The Conjecture is equivalent to \mbox{$\forall n \Lambda_n$},
where $\Lambda_n$ denote the statement}
\[
\forall x_1,\ldots,x_n \in \Z ~\exists y_1,\ldots,y_n \in \Z
\]
\[
\bigl(2^{\textstyle 2^{n-1}}<|x_1| \Longrightarrow (|x_1|<|y_1| \vee \ldots \vee |x_1|<|y_n|)\bigr) ~\wedge
\]
\[
\bigl(\forall i,j,k \in \{1,\ldots,n\}~(x_i+x_j=x_k \Longrightarrow y_i+y_j=y_k)\bigr) ~\wedge
\]
\[
\forall i,j,k \in \{1,\ldots,n\}~(x_i \cdot x_j=x_k \Longrightarrow y_i \cdot y_j=y_k)
\]
\begin{sloppypar}
\begin{lemma}\label{lem3}
For all positive integers $n$, $m$ with \mbox{$n \leq m$}, if the statement $\Lambda_n$ fails for
\mbox{$(x_1,\ldots,x_n) \in {\Z}^n$} and \mbox{$2^{\textstyle 2^{m-1}}<|x_1| \leq 2^{\textstyle 2^m}$},
then the statement $\Lambda_m$ fails for \mbox{$(\underbrace{x_1,\ldots,x_1}_{m-n+1 {\rm ~times}},x_2,\ldots,x_n) \in {\Z}^m$}.
\end{lemma}
\end{sloppypar}
\vskip 0.1truecm
\par
By the Corollary and Lemma~\ref{lem3}, the Conjecture is equivalent to \mbox{$\forall n \Psi_n$},
where $\Psi_n$ denote the statement
\[
\forall x_1,\ldots,x_n \in \Z ~\exists y_1,\ldots,y_n \in \Z
\]
\[
\Bigl(2^{\textstyle 2^{n-1}}<|x_1|={\rm max}\bigl(|x_1|,\ldots,|x_n|\bigr) \leq 2^{\textstyle 2^n}
\Longrightarrow \bigl(|x_1|<|y_1| \vee \ldots \vee |x_1|<|y_n|\bigr)\Bigr) ~\wedge
\]
\[
\Bigl(\forall i,j,k \in \{1,\ldots,n\}~(x_i+x_j=x_k \Longrightarrow y_i+y_j=y_k)\Bigr) ~\wedge
\]
\[
\forall i,j,k \in \{1,\ldots,n\}~(x_i \cdot x_j=x_k \Longrightarrow y_i \cdot y_j=y_k)
\]
In contradistinction to the statements $\Lambda_n$, each true statement~$\Psi_n$ can be confirmed
by a brute-force search in a finite amount of time.
\vskip 0.2truecm
\par
Let $T_n$ denote the set of all integer tuples \mbox{$(a_1,\ldots,a_n)$} for which
there exists a system \mbox{$S \subseteq E_n$} such that \mbox{$(a_1,\ldots,a_n)$} solves $S$
and $S$ has at most finitely many solutions in integers \mbox{$x_1,\ldots,x_n$}.
If \mbox{$(a_1,\ldots,a_n) \in T_n$}, then \mbox{$(a_1,\ldots,a_n)$} solves the system
\begin{displaymath}
\left\{
\begin{array}{rcl}
x_{i} &=& 1~~~~~~({\rm all~} i \in \{1,\ldots,n\} {\rm ~with~} a_i=1) \\
x_i+x_j &=& x_k~~~~~({\rm all~} i,j,k \in \{1,\ldots,n\} {\rm ~with~} a_i+a_j=a_k) \\
x_i \cdot x_j &=& x_k~~~~~({\rm all~} i,j,k \in \{1,\ldots,n\} {\rm ~with~} a_i \cdot a_j=a_k)
\end{array}
\right.
\end{displaymath}
which has only finitely many solutions in integers \mbox{$x_1,\ldots,x_n$}.
\newpage
\begin{theorem}\label{the3}
The Conjecture is true for \mbox{$n \leq 3$}.
\end{theorem}
\begin{proof}
$T_1=\{0,1\}$. $T_2$ consists of the pairs \mbox{$(0,0)$}, \mbox{$(1,1)$}, \mbox{$(-1,1)$},
\mbox{$(0,1)$}, \mbox{$(1,2)$}, \mbox{$(2,4)$} and their permutations. $T_3$ consists of the triples
\vskip 0.2truecm
\centerline{$(0,0,0)$, $(1,1,1)$,}
\vskip 0.2truecm
\centerline{$(-1,-1,1)$, $(0,0,1)$, $(1,1,-1)$, $(1,1,0)$, $(1,1,2)$, $(2,2,1)$, $(2,2,4)$, $(4,4,2)$,}
\vskip 0.2truecm
\centerline{$(1,-2,-1)$, $(1,-1,0)$, $(1,-1,2)$, $(1,0,2)$, $(1,2,3)$, $(1,2,4)$,}
\vskip 0.2truecm
\centerline{$(2,4,-2)$, $(2,4,0)$, $(2,4,6)$, $(2,4,8)$, $(2,4,16)$,}
\vskip 0.2truecm
\centerline{$(-4,-2,2)$, $(-2,-1,2)$, $(3,6,9)$, $(4,8,16)$}
\vskip 0.2truecm
\noindent
and their permutations.
\end{proof}

\noindent
Apoloniusz Tyszka\\
Technical Faculty\\
Hugo Ko\l{}\l{}\k{a}taj University\\
Balicka 116B, 30-149 Krak\'ow, Poland\\
E-mail address: \url{rttyszka@cyf-kr.edu.pl}
\end{document}